\documentclass[a4paper,english,numberwithinsect]{llncs}
\usepackage{microtype,mathtools} 
\usepackage{svg}
\usepackage{dirtytalk}
\usepackage{hyperref}
\usepackage{lineno}
\usepackage{amssymb}

\newtheorem{observation}[theorem]{Observation}

\DeclarePairedDelimiter{\ceil}{\lceil}{\rceil}

\newcommand*\samethanks[1][\value{footnote}]{\footnotemark[#1]}

\begin{document}
\title{Star-Forest Decompositions of Complete
Graphs}

\author{Todor Anti\'c \inst{1} \thanks{The research was conducted during a scholarship provided by Vi\v segrad Fund.}\and
Jelena Gli\v{s}i\'c \inst{1} \samethanks\and
Milan Milivoj\v{c}evi\'{c} \inst{1,2}}

\authorrunning{T. Anti\'{c} et al.}

\institute{Faculty of Mathematics and Physics, Charles University
\\ \email{todor@kam.mff.cuni.cz, jglisic@matfyz.cz, milivojcevic@matfyz.cz}\and
Faculty of Mathematics, Natural Sciences and Information Technologies, University of Primorska
}

\maketitle            

\begin{abstract}  
We deal with the problem of decomposing a complete geometric graph into plane star-forests. In particular, we disprove a recent conjecture by Pach, Saghafian and Schnider by constructing for each $n$ a complete geometric graph on $n$ vertices which can be decomposed into $\frac{n}{2}+1$ plane star-forests. Additionally we prove that for even $n$, every decomposition of complete abstract graph on $n$ vertices into $\frac{n}{2}+1$ star-forests is composed of a perfect matching and $\frac{n}{2}$ star-forests with two edge-balanced components, which we call broken double stars.

\end{abstract}

\section{Introduction}

A classic question asked in graph theory is the following:
\say{Given a graph $G$, what is the minimal number of subgraphs with property $P$ that the edges of $G$ can be partitioned into?}
Historically, this question was asked for abstract graphs and property $P$ was replaced with forests, trees, complete bipartite graphs and many more \cite{akbari,lonc,shyu}.
Similar questions can be asked about graphs drawn in the plane or on any other surface. Here we want to decompose a complete graph on the surface into subgraphs that have a certain geometric property in addition to the property $P$. Answering such questions is a similar, but separate research direction that has been pursued by many authors in discrete geometry and graph drawing communities.

A \emph{geometric} graph is a graph drawn in the plane, with vertices represented by points in general position and edges as straight line segments between them. 

Recently, there has been a lot of work done on decomposing geometric graphs into planar subgraphs of a special kind, such as trees, stars, double stars  etc. \cite{HazimMichan,Bose}. This paper will be concerned with decomposing complete geometric graphs into plane star-forests. 
A \emph{star} is a connected graph on $k$ vertices with one vertex of degree $k-1$ (center) and $k-1$ vertices of degree $1$. Our definition allows a graph that has two vertices and a single edge to be a star but it is not clear what the center should be. In this case we define the center to be one of the endpoints of the edge, and this choice is arbitrary. The definition also allows that a single vertex is a star and in this case this vertex is also the center of said star. A \emph{star-forest} is a forest whose every connected component is a star. It is easy to observe that a complete graph $K_n$ can be decomposed into $n-1$ stars. Furthermore, $K_n$ cannot be decomposed into less than $n-1$ stars \cite{akiyama}. In the same paper, Akiyama and Kano proved that $K_n$ can be decomposed into at most $\lceil \frac{n}{2} \rceil+1$ star-forests and that this bound is tight.

The story is different for complete geometric graphs. 
To the best of our knowledge, the first mention of star-forest decompositions was made by Dujmovi\'{c} and Wood in \cite{dujmovic}. They asked if one can decompose a complete geometric graph on $n-1$ vertices, whose vertices form a convex polygon into less than $n-1$ star-forests. 

Recently, this question was ansewered in the negative by Pach, Saghafian and Schnider \cite{pach}. They showed that a complete geometric graph whose vertices form a convex polygon cannot be decomposed into fewer than $n-1$ plane star-forests. 
In the same paper, the authors posed the following question, which they considered the most important one in this direction. 

\begin{question}
    What is the minimal number of plane star-forests that a complete geometric graph can be decomposed into?
\end{question}

Based on their findings they made the following conjecture: 

\begin{conjecture}[\cite{pach}]\label{conj: wrong}
Let $n\ge1$. There is no complete geometric graph with $n$ vertices that can be decomposed into fewer than $\ceil{3n/4}$ plane star-forests.
\end{conjecture}

The main aim of this note is to answer this conjecture in the negative. 
The authors in \cite{pach} give a special configuration of $n=4k$ points and construct a simple decomposition into $3n/4$ plane star-forests. Motivated by this example, we first describe a method generalizing it. This is the content of Theorem \ref{thm: glavna} 
Then we provide a point set on $n=6$ points which can be decomposed into $\frac{2n}{3}=4$ star-forests, disproving the conjecture. 
We then improve the bound further by constructing complete geometric graphs which can be decomposed into $\frac{n}{2}+1$ plane star-forests, which is best possible. This is the content of Theorem \ref{thm: optimalna}.  
Attacking this problem raised some further questions regarding decompositions of abstract complete graphs into star-forests. Mainly, our computations \cite{github} have shown us that for $n=6,8$, the decomposition of $K_n$ into $\frac{n}{2}+1$ star-forests is unique in a certain sense. In each fitting decomposition one star-forest was a perfect matching on $\frac{n}{2}$ vertices while the other $\frac{n}{2}$ star-forests were edge balanced, spanning and had centers at endpoints of an edge of the matching. We call such a decomposition a broken double stars decomposition. 
We prove that for $n$ even, every decomposition of the abstract complete graph $K_n$ into $ \frac{n}{2}+1$ star-forests is a broken double stars decomposition. This is the content of Theorem \ref{thm: unique star-forests}. 

\section{Decompositions of Complete Graphs into Star-Forests}

In this section our goal is to define a unique construction of star-forest decompositions of the complete graph.  These star-forests will be a byproduct of a decomposition of $K_n$ into special trees, called \emph{double stars}. A double star is a graph composed of two vertex-disjoint stars, whose centers are joined by an edge. For an even $n$, the double star decomposition of $K_n$ is obtained in a following way. Let $M$ be a perfect matching in $K_n$. Then, for each edge $e\in M$ we create a double star by connecting each endpoint of $e$ to $\frac{n-2}{2}$ vertices of $K_n$ in such a way that we do not obtain a cycle. This results in a decomposition of the edge set into $\frac{n}{2}$ double stars.
From this we can define a decomposition of $K_n$ into star-forests in the most natural way. One forest is a matching on $k$ edges, and each of the other $k$ forests is composed of two stars with $k-1$ edges each, whose centers are the endpoints of an edge of the matching. This construction was described in \cite{akiyama} and again in \cite{pach}. For a visual explanation see Figure \ref{fig: broken_double_stars}. We will call any such a decomposition of $K_n$ the \emph{broken double stars} decomposition. We now formally state the main result of this section. 

\begin{figure}[h]
    \centering
    \includegraphics[width=0.25\linewidth]{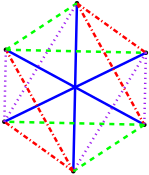}
    \caption{A broken double star decomposition of $K_6$}
    \label{fig: broken_double_stars}
\end{figure}

\begin{theorem}\label{thm: unique star-forests}
    Let $n=2k$ be an even integer. Then any decomposition of $K_n$ into $k+1$ star-forests is a broken double star decomposition.
\end{theorem}

But before proving this we need a couple of building blocks. We first note that we can assume that a decomposition of $K_{2k}$ into $k+1$ star-forests cannot contain a star forest with a single component. This is because removing the center of this component and all of the adjacent edges would result in a decomposition of $K_{2k-1}$ into $k< \lceil \frac{2k-1}{2}\rceil+1$ star forests. Which is impossible by the result of Akiyama and Kano \cite{akiyama}.

\begin{lemma}\label{lemma: matchingimpliesspanning}
    Let $n=2k$ be an even integer and let $F_0,\dots, F_k$ be a decomposition of $K_n$ into $k+1$ star-forests. If $F_0$ is a matching then every other star-forest is spanning and consists of two components. 
\end{lemma}
\begin{proof}
    Assume that we have such a decomposition and that one of the star-forests has at least three components. Then it can have at most $2k-3$ edges. Then the number of edges covered by the decomposition is at most $k+2k-3 + (k-1)(2k+2) < k(2k-1) = \binom{2k}{2}$. 
\end{proof}

\begin{lemma}\label{lemma: matchinimpliesbrokendoublestars}
    Let $n=2k$ be an even integer and let $\mathcal{F}$ be a decomposition of $K_n$ into $k+1$ star-forests. If some $F\in\mathcal{F}$ is a matching then this decomposition is a broken double stars decomposition. 
\end{lemma}

\begin{proof}
    Let us label the vertices of $K_n$ as $\{v_1,v_2,\dots,v_n\}$. Assume that we have a decomposition of $K_n$ into star-forests such that one of the forests is a matching. Assume without loss of generality that the edges of the matching are given by $\{v_i,v_{i+k}\}$ for $1\le i\le k$. Now we will prove that for each forest $F' \neq F$ there exists an $i$ such that the centers of the stars in  star-forest $F'$ are exactly $v_i,v_{i+k}$. Assume for contradiction that there is a star-forest $F^*$ in the decomposition whose stars have centers $v_i$ and $v_j$ for some $j\neq i \pm k$. Then consider the star-forest with center in $v_{i+k}$ and let $u$ be the center of the other star in this forest. Assume that $F^*$ is spanning, and thus that $\{v_{i+k},v_j\} \in E(F^*)$ and without loss of generality $\{u,v_j\}\in E(F^*)$. But, since $\{v_i,v_{i+k}\}$ is an edge in the matching it means that the forest with centers $v_{i+k},u$ cannot be spanning, contradicting Lemma \ref{lemma: matchingimpliesspanning}. Lastly, we need to prove that inside of each star-forest, stars have equal number of edges. For this we will assume that the forest with centers in $v_i, v_{i+k}$ does not have stars with equal number of edges. Assume that in this star-forest $v_i$ is connected to vertices $\{v_{i+1},\dots,v_{i+k-1},v_{i+k+1}\}$ and that $v_{i+k}$ is connected with $\{v_{i+k+2},\dots, v_{2k}\}$. Then both of the edges $\{v_{i+k}, v_{i+1}\}$ and $\{v_{i+k},v_{i+k+1}\}$ need to be used by the star-forest with appropriate centers, which is clearly impossible. 
\end{proof}

\noindent The following observation has been checked computationally.
\begin{observation}\label{obs: truefor6}
    Any decomposition of $K_6$ into $4$ star-forests is a broken double stars decomposition.
\end{observation}

\noindent Finally we proceed with a proof of Theorem \ref{thm: unique star-forests}
\begin{proof}[Proof of Theorem \ref{thm: unique star-forests}]
We will proceed by induction on $k$. For the base case $k=3$ the claim holds by Observation \ref{obs: truefor6}. Now assume that it holds for $k-1$. Suppose that we have a decomposition of $K_{2k}$ into $k+1$ star-forests. If any star-forest has a single component which is not a single vertex, removing the center of that component leaves us with a decomposition of $K_{2k-1}$ into $k$ star-forests, which is impossible as $k < \lceil \frac{2k-1}{2}\rceil +1$. Further if each star-forest has at least three components, we cannot cover all of the edges since each component can have at most $2k-3$ edges and $(k+1)(2k-3)<\binom{2k+1}{2}$. Therefore there is at least one star-forest with two components. Consider the graph obtained by removing the centers (call them $c,d$ of such a star-forest from $K_{2k}$. We are then left with a decomposition of $K_{2k-2}$ into $k$ star-forests. By the inductive hypothesis, this must be a broken double star decomposition. One of the edges removed was the one between the two centers, call it $e$.
Assume that $e$ belongs to a forest with centers $a,b$. Then the edges between $c,d$  and $a,b$ cannot be in the same forest as $e$. So assume that two of these edges, $\{a,c\}$ and $\{d,b\}$ belong to another forest with centers $a',b'$. Then, without loss of generality, the edges $\{a,a'\},\{b,b'\}$ belong to the forest with centers $a,b$ and the edges $\{a,b'\},\{b,a'\}$ must belong to a forest different than the one with centers $a',b'$, contradicting the inductive assumption. Therefore, $e$ must extend the matching in the decomposition of the smaller graph and the result follows by Lemma \ref{lemma: matchinimpliesbrokendoublestars}. 
\end{proof}

\section{Decomposing Complete Geometric Graphs into Plane Star-Forests}

Firstly, we will describe a generalization of the method from \cite{pach} and use it to construct a counterexample. After this we  construct a complete geometric graph on $2k$ vertices which can be decomposed into $k+1$ plane star-forests. From now on, we will write $GP$ for a complete geometric graph whose underlying point set is $P\subseteq\mathbb{R}^2$.

\begin{theorem} \label{thm: glavna}
    Let $c\in (1/2,1)$ be a constant. If there is a complete geometric graph on $n_0$ points which can be partitioned into $cn_0$ plane star-forests, in such a way that each vertex is a center of at least one tree, then for each integer $k\ge1$, there exists a complete geometric graph on $kn_0$ points that can be  partitioned into $ckn_0$ plane star-forests. 
\end{theorem}

\begin{proof}
Let $S$ be the underlying point set of the original complete geometric graph and let $k>1$ be an integer. Label the points in $S$ by $a_1,\dots,a_{n_0}$. Now, replace each $a_i$ by a set $A_i=\{a_i^1,\dots,a_i^k\}$ of $k$ points in general position in such a way that if we choose $b_1,\dots,b_{n_0}$ where $b_i\in A_i$, we obtain a point set of the same order type as $S$. Call the new point set $S^k$.
Now if $F_1,\dots,F_{cn_0}$ is the decomposition of $GS$ into plane star-forests, from this, we will obtain the decomposition of $GS^k$ into $c(kn_0)$ plane star-forests. 
Let $a_j$ be the center of a star in $F_i$. We will construct $k$ new stars with centers in $a_j^1,\dots,a_j^k$. Start with $a_j^1$, add to it all of the edges of the form $\{a_j^1,a_j^l\}$ that were not already used (in the case of $a_j^1$, none were used). Now for each edge of the form $\{a_j,a_{j'}\}$ in $F_i$, add all of the edges from $a_j^1$ to the vertices in $A_{j'}$. Continue doing this for each vertex $a_j^l$, where $l\in\{1,2,\dots,k\}$.
We do this for each star in $F_i$ and for each forest in the original decomposition. The result of this process is $cn_0$ families of star-forests, each of size $k$. And the planarity of the star-forests follows from the definition of the point set $S^k$. To see this, assume that a tree in the new decomposition has an intersection. Then the intersection is between edges whose $4$ vertices are in different $A_i$'s. But if this was the case, then a choice of transversal that includes this $4$ vertices would induce a crossing inside the original decomposition of $GS$.  
\end{proof}

We note that the assumption that each point is a center of at least one forest is crucial as otherwise the star-forests constructed in the proof do not cover all of the edges.

      \begin{figure}[ht]
            \centering
             \begin{tabular}{cc}
     \includegraphics[scale=0.3]{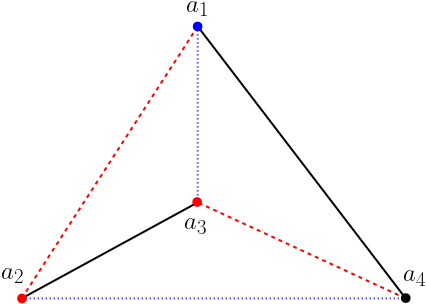} \textcolor{white}{aaaa}
    \includegraphics[scale=0.3]{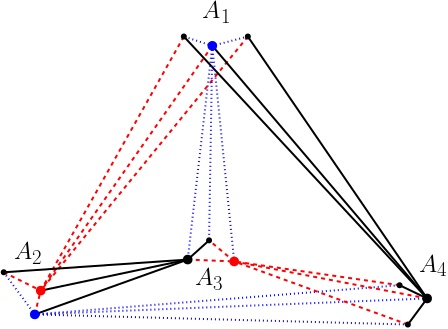} 
    
    \end{tabular}
            \caption{A complete geometric graph on $4$ vertices decomposed into three plane star-forests and the corresponding graph on $12$ vertices with the decomposition into $9$ star-forests (only $4$ are drawn for readability). Each vertex of the point set on the left has been used as a center of one tree and colored accordingly.}
            \label{fig: 4pointscorr}
        \end{figure}

While Theorem \ref{thm: glavna} gives us a nice way of constructing infinitely many complete geometric graphs that can be partitioned into few plane star-forests, we still need concrete small examples to be able to produce the infinitudes. One example was given by the authors in \cite{pach} and can be found in Figure \ref{fig: 4pointscorr}. This example motivated Conjecture \ref{conj: wrong}. We proceed in a similar fashion.

\begin{lemma}
    There exists a configuration of $6$ points in the plane which can be partitioned into $4$ plane star-forests in such a way that each point is a center of at least one star.
\end{lemma}
\begin{proof}
    We consider a configuration of $6$ points which is crossing-minimal according to \cite{welzl}. We decompose the graph into $4$ star-forests as in Figure \ref{fig: 6points}.   The graph has thus been decomposed into three 2-component star-forests colored blue, red and black and one 3-component forest colored purple.
\end{proof}
Now, using the point set of $n_0=6$ elements from the above lemma, which can be decomposed into $2n_0/3=4$ star-forests, we obtain as an easy corollary a family of point sets on $n=6k$ points which can be decomposed into $2n/3$ star-forests, thus disproving Conjecture \ref{conj: wrong}. We state this formally below.

\begin{corollary}
    For every $n$ divisible by $6$, there exists a geometric graph on $n$ vertices which can be decomposed into $2n/3$ plane star-forests. 
\end{corollary}

\begin{figure}
            \centering
            \includegraphics[scale=0.55]{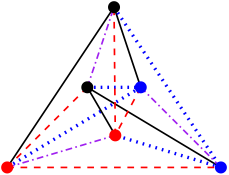}
            \caption{A complete geometric graph on $6$ vertices decomposed into four star-forests, vertices are colored the same as trees whose centers they are.}
            \label{fig: 6points}
        \end{figure}

For every $k\in \mathbf{N}$ we construct a point set on $2k$ points that can be decomposed into $k+1$ plane star-forests. By Theorem \ref{thm: unique star-forests}, one star-forest in the decomposition will be a matching. So our approach will be to first construct the matching as an arrangement of line segments in the plane and then use it to construct the other star-forests. We will say that the arrangement of $k$ line segments is \emph{SF-extendable} if the geometric graph with underlying point set consisting of the endpoints of the line segments admits a decomposition into $k+1$ plane star-forests, one of which is the matching given by the arrangement. We say that two line segments are in a \emph{stabbing} position if the convex hull of their endpoints is a triangle. If $s=ab$ and $l=cd$ are two line segments in a stabbing position and the convex hull of $\{a,b,c,d\}$ contains $c$ or $d$ in the interior, we say that \emph{$l$ stabs $s$}.

We now provide a necessary condition for an arrangement to be SF-extendable. 

\begin{lemma} \label{lem: necessary}
    Let $\mathcal{L}$ be an arrangement of line segments in the plane. If $\mathcal{L}$ is SF-extendable then every pair of segments from $\mathcal{L}$ is in a stabbing position. 
\end{lemma}

\begin{proof}
    Assume on the contrary that there are two line segments $ab$ and $cd$ which are not in a stabbing position. Then the convex hull of $\{a,b,c,d\}$ forms a convex quadrangle. Assume that the cyclic ordering of vertices along the convex hull is $(a,b,c,d)$. The two diagonals of this quadrangle intersect. Assume that $\mathcal{L}$ is SF-extendable. If the star-forest with centers in $a$ and $b$ contains the edges $\{a,c\}$ and $\{b,d\}$. Then the star-forest with centers in $c$ and $d$ is not planar, contradicting the assumption that $\mathcal{L}$ is SF-extendable. If the star-forest with centers in $a$ and $b$ contains the edges $\{a,d\}$ and $\{b,c\}$ then it is not planar, again contradicting the assumption that $\mathcal{L}$ is SF-extendable.
\end{proof}

Let $L_1$ be a segment of length $1$ with center at the origin. Call its left endpoint $a_1$ and the right endpoint $b_1$. Let $a_1,a_2$,\dots,$a_k,b_1$ be vertices of some convex $(k+1)$-gon $P$ such that $L_1$ is an edge of $P$ and each $a_i$ for $i>1$ is placed inside of the top left quadrant of the plane. 
We will now construct line segments $L_2,\dots L_k$ with endpoints $a_i,b_i$ respectively. For each $i>1$ place each $b_i$ in the intersection of interiors of all of the triangles $(a_l,b_l,a_j)$ where $l<j<i$ and the top right quadrant of the plane. We call such a line arrangement a $k$-staircase. See Figure \ref{fig: stairs} for examples of a $3$-staircase and a $4$-staircase and their extension into a star-forest decomposition. 

\begin{theorem}\label{thm: staircase}
    For each $k\ge 1$, there exists an SF-extendable arrangement of $k$ line segments.
\end{theorem}
\begin{proof}
    We will show that a $k$-staircase is SF-extendable. The star-forest with centers in $a_i,b_i$ has the edgeset $$\{\{a_i,a_j\}: j>i\}\cup \{\{a_i,b_k\}: k<i\} \cup \{\{ b_i,b_j\}: j>i\} \cup \{\{b_i,a_k\}: k<i\}.$$
    It is clear that every edge will be covered by this decomposition. Now we will check planarity of the forest with centers $a_i,b_i$. Edges of the form $\{a_i,a_j\}$ for $j>i$ cannot cross edges of the form $\{b_i,b_j\}$ since $a$ points are divided from $b$ points by the $y$-axis. Edges of the form $\{a_i,a_j\}$ cannot cross edges of the form $\{b_i,a_k\}$ where $k<i\le j$ since $a_j,a_k$ are in different half-planes determined by a line through $a_i,b_i$.  The other cases are similar. 
\end{proof}

\begin{theorem}\label{thm: optimalna}
    For each $n$ there exists a complete geometric graph on $n$ points which can be decomposed into $\lfloor \frac{n}{2}\rfloor +1$ plane star-forests. 
\end{theorem}
\begin{proof}
    If $n=2k$ is even, the complete geometric graph is given by the $k$-staircase. In the case $n=2k-1$, take a single point away from the $k$-staircase and the resulting complete geometric graph gives the result. 
\end{proof}

\begin{figure}
    \centering
    \includegraphics[scale=0.23415]{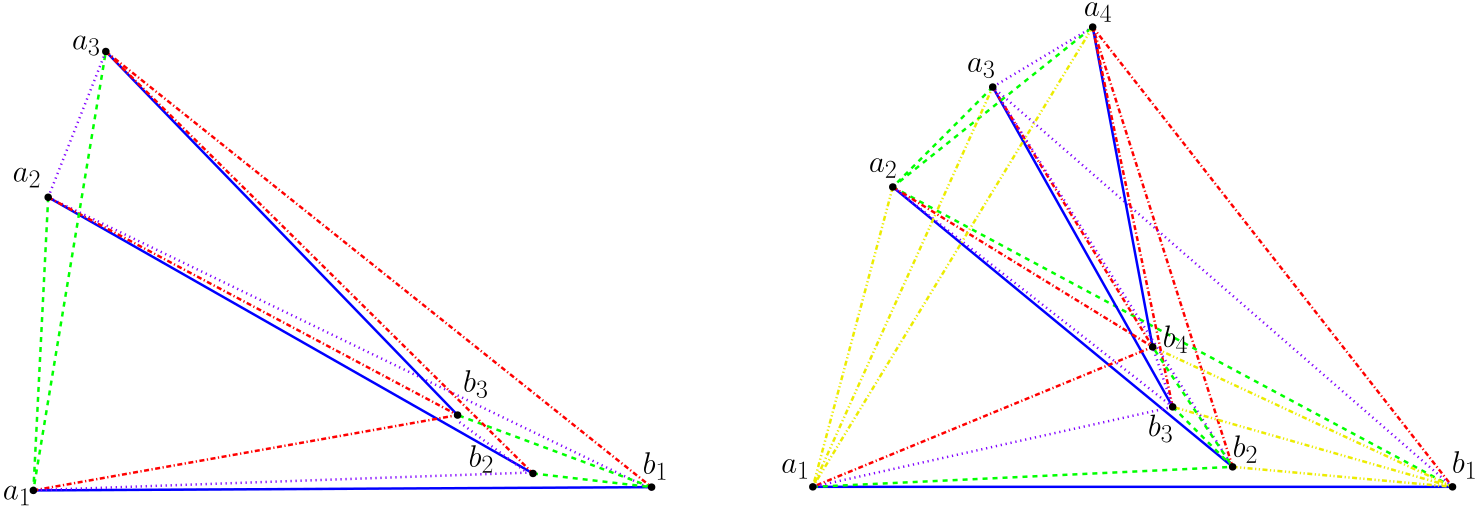}
    \caption{Point sets in 3-staircase and 4-staircase configurations.}
    \label{fig: stairs}
\end{figure}

The $k$-staircase produces a point set with a fairly large convex hull. So a natural question is to ask if this is in fact necessary or can we find point sets with a smaller convex hull. The following construction shows that we can in fact do this. 

We will again construct an arrangment of line segments. Start with a segment of length 1 centered at the origin, and call its endpoints $a_1,b_1$ as before. Now place points $a_2,\dots ,a_k$ in such a way that each $a_i$ has a smaller $x$ and larger $y$ coordinate than $a_{i-1}$. Then place $b_2,\dots,b_k$ so that they obey the same conditions as in the construction of the $k$-staircase. This way we obtain an arrangement of line segments $L_1,\dots, L_k$ which is SF-extendable which we will call $k$-comet. To prove this fact we can use the same star-forest decomposition as we did in the proof of Theorem \ref{thm: staircase}. However, where $k$-staircase defines a point set on $2k$ points with convex hull of size $k+1$, the $k$-comet defines a point set on $2k$ points with convex hull of size $3$, see Figure \ref{fig: comet}. 

In fact it is not hard to construct such a point set with a convex hull of any size between $3$ and $k+1$. First note that the $k$-comet is "obtained" from a $k$-staircase by turning a convex polygonal line into a concave one. We say "obtained" since one also needs to adjust the positions of $b_i$'s in the construction. Now we can do the same thing but turn only an initial segment of the convex polygonal line into a concave segement and obtain a smaller convex hull.
Even further, one can make the following observation.

\begin{observation}\label{obs: sizeofconvhulls}
Let $k\ge 2$. If $G$ is a complete geometric graph on $n=2k$ vertices which can be decomposed into $k+1$ plane star-forests, then the size of the convex hull of $V(G)$ is at most $k+1$.
\end{observation}
\begin{proof}
    By Theorem \ref{thm: unique star-forests} we know that one of the star-forests will be a perfect matching. If $V(G)$ has a convex hull of size at least $k+2$, then the matching needs to use at least $2$ edges of the convex hull. But then the matching cannot be SF-extendable by Lemma \ref{lem: necessary} since the edges of the convex hull are never in a stabbing position. 
\end{proof}

Based on this we make the following conjecture. 
\begin{conjecture} \label{conj: sizeofconvhulls}
    Let $n\ge 3$ be odd, $G$ a complete geometric graph on $n$ vertices which can be decomposed into $\lceil\frac{n}{2}\rceil+1$ plane star-forests. Then the convex hull of $V(G)$ has size at most $\lceil\frac{n}{2}\rceil+1$.  
\end{conjecture}

What is not clear is if one can construct point sets with convex hull of size $k+1$ that are not a $k$-staircase. Our computations on point sets with $6$ and $8$ points found no such point sets. Thus we make the following conjecture. 

\begin{conjecture}\label{conj: ours}
    If a complete geometric graph $G$ on $2k$ points admits a decomposition into $k+1$ plane star-forests and the size of the convex hull of $V(G)$ is $k+1$ then $V(G)$ can be described as a $k$-staircase.
\end{conjecture}

\begin{figure}
    \centering
    \begin{tabular}{cc}
    
    \includegraphics[scale=0.52]{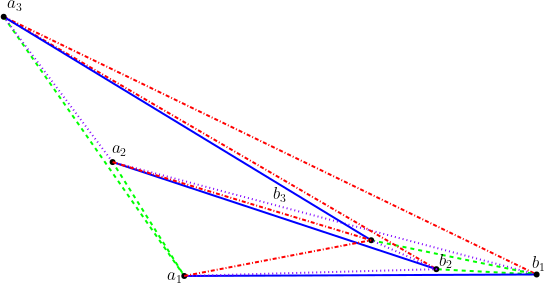} &
   
    \end{tabular}
    \caption{Point set in 3-comet configuration.}
    \label{fig: comet}
\end{figure}

\section{Computing Plane Star-Forest Decompositions on Small Point Sets}

Using a simple computer search, we managed to find all point sets on $6$ points that can be decomposed into $4$ plane star-forests and all point sets on $8$ points that can be decomposed into $5$ plane star-forests. Out of the $16$ order types on $6$ points,  which can be found on \cite{pointsets}, we have found decompositions which satisfy the requirements from Theorem \ref{thm: glavna}  for $6$ of them. Those point sets and corresponding partitions can be seen in Figure \ref{fig: decompostitions}. Out of $3315$ order types on $8$ points, we have found such decompositions for $411$ of them. 
The code is available at \cite{github}. We plan to continue improving the code to be able to perform the search on bigger point sets. Currently, the generation of appropriate decompositions is very slow, and since the number of edges and corresponding decompositions grows very fast, we are not able to perform the checks for bigger point sets. Of course, for even bigger point sets, a completely different approach would be needed, as the number of different order types for $n\ge 12$ becomes too large to handle. For this, using SAT solvers might be useful to try and resolve Conjectures \ref{conj: sizeofconvhulls} and \ref{conj: ours}.

\begin{figure}[h]
    \centering
    \begin{tabular}{ccc}
        \includegraphics[width=0.3\linewidth]{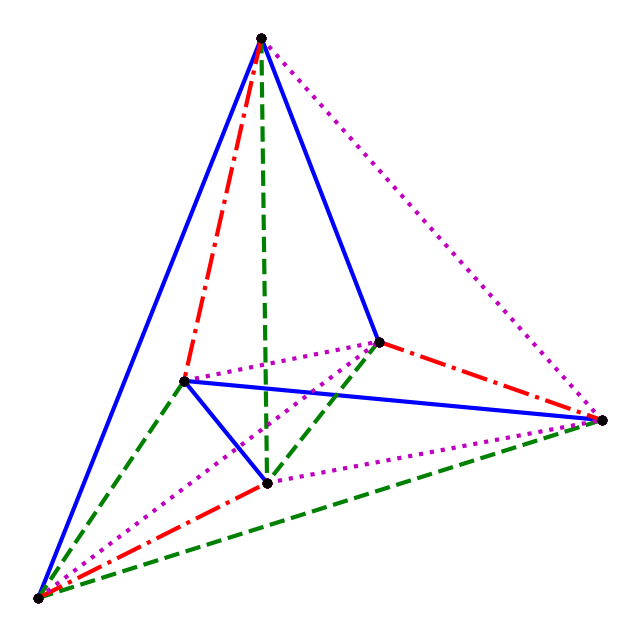} & \includegraphics[width=0.3\linewidth]{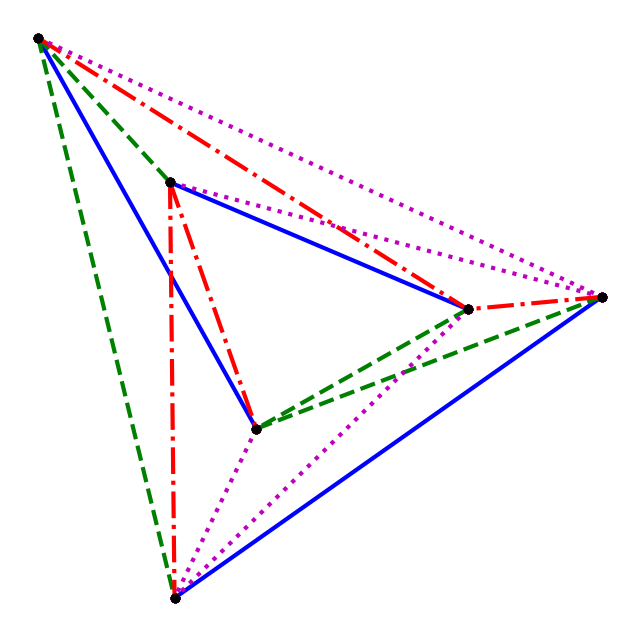} & \includegraphics[width=0.3\linewidth]{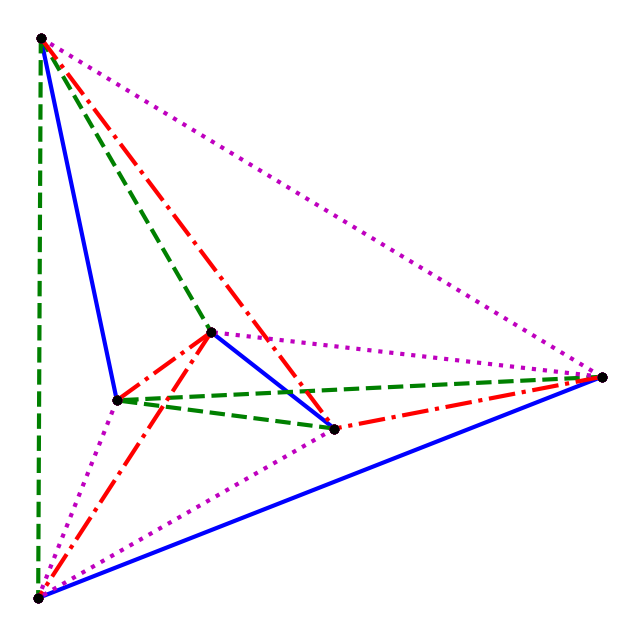} \\[6pt]
        \includegraphics[width=0.3\linewidth]{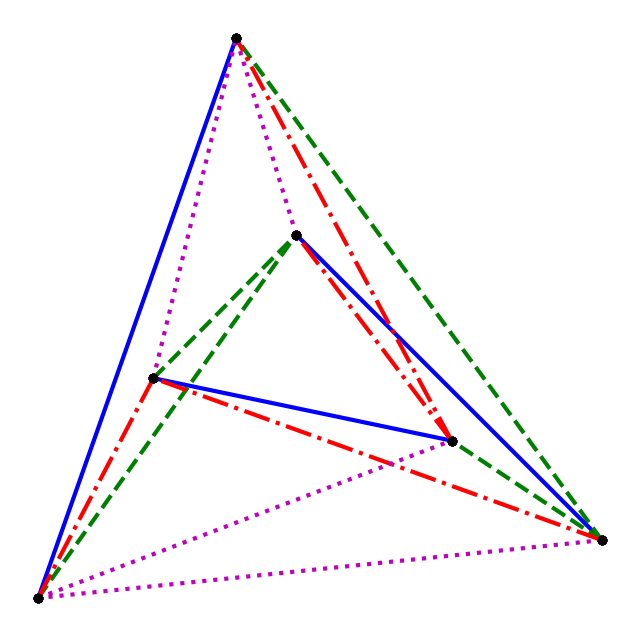} & \includegraphics[width=0.3\linewidth]{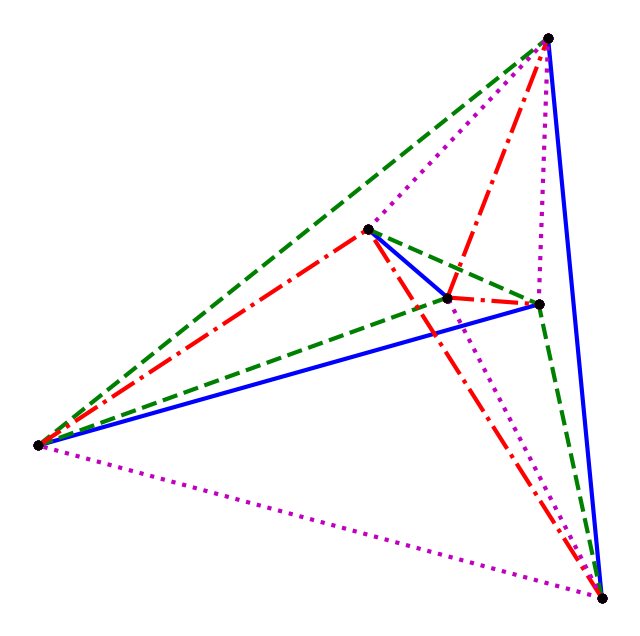} & \includegraphics[width=0.3\linewidth]{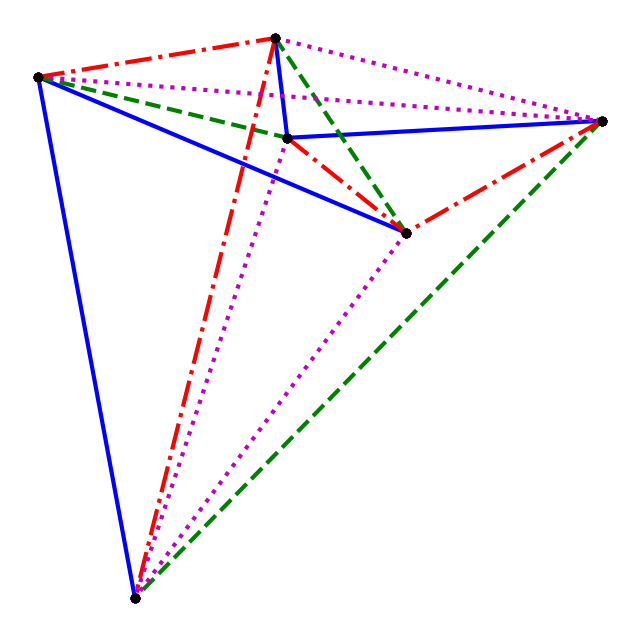} \\[6pt]
    \end{tabular}
    \caption{Star-forest decompostions of the point sets that admit them.}
    \label{fig: decompostitions}
\end{figure}

\section{Further Research and Open Questions}

While our construction is optimal in the sense that it minimizes the size of minimal decomposition into plane star-forests, it is still unclear what characterizes complete geometric graphs that admit such a decomposition. We have made no progress towards solving this problem but computational results show that these point sets can be quite diverse. However, a possibly easier question to answer is the following. 

\begin{question}
    What are sufficient conditions for an arrangement of line segments to be SF-extendable?
\end{question} 

Lemma \ref{lem: necessary} gives us a necessary condition, but it is not hard to see that this is not sufficient to guarantee SF-extendability. 

We also note that there is an interesting variation of the original problem that we have not explored yet but where our approach from Theorem \ref{thm: glavna} can also be used. We define a \emph{k-star-forest} to be a star-forest with at most $k$ components. Authors in \cite{pach} proposed the following conjecture: 

\begin{conjecture}\cite{pach}\label{conj: maybecorrect}
    The number of plane $k$-star-forests needed to decompose a complete geometric graph is at least $\frac{(k+1)n}{2k}$. 
\end{conjecture}

Our example does not show anything regarding Conjecture \ref{conj: maybecorrect}. But, it is not hard to see that the construction from Theorem \ref{thm: glavna} preserves the maximal number of components among all forests. Thus, we believe a similar approach could be used to attack this conjecture. 

\section*{Acknowledgments}
This work is supported by project 23-04949X of the Czech Science Foundation
(GA\v{C}R). We also thank Pavel Valtr and Jan Kyn\v{c}l who proposed the problem to us and gave a lot of useful comments and suggestions during the combinatorial problems seminar at Charles University. 
\bibliographystyle{splncs04}
\bibliography{eurocg24_example}

\end{document}